\documentclass[11pt]{article}
\usepackage{amsmath}
\usepackage{bbm}
\usepackage{amsthm}
\usepackage{txfonts}
\usepackage{indentfirst,latexsym,bm}

\newfont{\cmtt}{cmtt10}

 \newtheorem{theorem}{Theorem}[section] \newtheorem{%
definition}[theorem]{Definition} \newtheorem{lemma}[theorem]{Lemma}
\newtheorem{corollary}[theorem]{Corollary} \newtheorem{remark}[theorem]{%
Remark}  %
\numberwithin{equation}{section}
\begin{document}

\begin{center}

\textbf{\Large SYNCHRONIZATION OF COUPLED STOCHASTIC SYSTEMS}

\vskip 0.2cm

\textbf{\Large WITH MULTIPLICATIVE NOISE}

\vskip 1cm

{\large ZHONGWEI SHEN and SHENGFAN ZHOU$^{*}$}

\vskip 0.2cm

{\small \textit{Department of Applied Mathematics, Shanghai Normal
University,}}

{\small \textit{Shanghai 200234, People's Republic of China}}

{\small \textit{shenzhongwei1985@hotmail.com}}

{\small \textit{$^*$zhoushengfan@yahoo.com}}

\vskip 0.2cm

{\large XIAOYING HAN}

{\small \textit{Department of Mathematics and Statistics, Auburn
University,}}

{\small \textit{Auburn, AL 36849, USA}}

{\small \textit{xzh0003@auburn.edu}}

\vskip 1cm

\begin{minipage}[c]{13cm}

\noindent {\small We consider the synchronization of solutions to
coupled systems of the conjugate random ordinary differential
equations (RODEs) for the $N$-Stratronovich stochastic ordinary
differential equations (SODEs) with linear multiplicative noise
($N\in \mathbb{N}$). We consider the synchronization between two
solutions and among different components of solutions under
one-sided dissipative Lipschitz conditions. We first show that the
random dynamical system generated by the solution of the coupled
RODEs has a singleton sets random attractor which implies the
synchronization of any two solutions. Moreover, the singleton sets
random attractor determines a stationary stochastic solution of the
equivalently coupled SODEs. Then we show that any solution of the
RODEs converge to a solution of the averaged RODE within any finite
time interval as the coupled coefficient tends to infinity. Our
results generalize the work of two Stratronovich SODEs in \cite{9}.

\vspace{5pt}


\vspace{5pt}

}

\end{minipage}
\end{center}

\section{Introduction}

The synchronization of coupled systems is a well known phenomenon in
both biology and physics.  It is also known to occur in many other
different fields. Descriptions of its diversity of occurrence can be
found in \cite {1, 2, 3, 4, 7, 12, 13, 14, 15, 17, 20}.
Synchronization of deterministic coupled systems has been
investigated mathematically in \cite{5, 11, 19} for autonomous
systems and in \cite{16} for nonautonomous systems. For coupled
systems of It\^{o} stochastic ordinary differential equations with
additive noise, Caraballo \& Kloeden proved its synchronization of
solutions in \cite{8}.

Let $(\Omega,\mathcal{F},\mathbb{P})$ be a probability space, where
\[
\begin{split}
\Omega =\big\{\omega \in {C}(\mathbb{R},\mathbb{R}):\omega (0)=0\big\}=C_0(%
\mathbb{R} ,\mathbb{R}),
\end{split}
\]
the Borel $\sigma $-algebra $\mathcal{F}$ on $\Omega $ is generated
by the compact open topology (see \cite{6,14}), and $\mathbb{P}$ is
the corresponding Wiener measure on $(\Omega,\mathcal{F})$. Define
$(\theta_{t})_{t\in\mathbb{R}}$ on $\Omega $ via
\[
\begin{split}
\theta _t\omega (\cdot )=\omega (\cdot +t)-\omega (t),\quad
t\in\mathbb{R},
\end{split}
\]
then $(\Omega ,\mathcal{F},\mathbb{P},(\theta_t)_{t\in\mathbb{R}})$
becomes an ergodic metric dynamical system.

Consider the following $N$-Stratonovich stochastic ordinary
differential equations (SODEs) in $\mathbb{R}^d$ ($d\in
\mathbb{N}$):
\begin{equation}\label{original-sode}
\begin{split}
dX_t^{(j)}=f^{(j)}(X_t^{(j)})dt+\sum\limits_{i=1}^mc_i^{(j)}X_t^{(j)}\circ
dW_t^{(i)},\quad j=1,\dots ,N,
\end{split}
\end{equation}
where $c_i^{(j)}\in \mathbb{R},$ $W_t^{(i)}$ are independent two-sided
scalar Wiener processes on $(\Omega ,\mathcal{F},\mathbb{P})$ for $i=1,\dots
,m$, and $f^{(j)},$ $j=1,...,N$, are regular enough to ensure the existence
and uniqueness of solutions and satisfy the one-sided dissipative Lipschitz
conditions
\begin{equation}\label{dissipative-condition}
\begin{split}
\Big\langle x_1-x_2,f^{(j)}(x_1)-f^{(j)}(x_2)\Big\rangle\leq
-L\|x_1-x_2\|^2,\quad j=1,\dots ,N
\end{split}
\end{equation}
on $\mathbb{R}^d$ for some $L>0$.

Set
\[
\begin{split}
x^{(j)}(t,\omega )=e^{-O_t^{(j)}(\omega )}X_t^{(j)}(\omega ),\quad t\in %
\mathbb{R},\quad \omega \in \Omega ,\quad j=1,\dots ,N,
\end{split}
\]
where
\[
\begin{split}
O_t^{(j)}=\sum\limits_{i=1}^mc_i^{(j)}e^{-t}\int_{-\infty }^te^\tau dW_\tau
^{(i)},\quad j=1,\dots ,N,
\end{split}
\]
are $N$ stationary Ornstein-Uhlenbeck processes which solve the following
Ornstein-Uhlenbeck stochastic differential equations, respectively,
\[
\begin{split}
dO_t^{(j)}=-O_t^{(j)}dt+\sum\limits_{i=1}^mc_i^{(j)}dW_t^{(i)},\quad
j=1,\dots ,N.
\end{split}
\]
Then SODEs \eqref{original-sode} can be transformed into the
following conjugate pathwise random ordinary differential equations
(RODEs)
\begin{equation}\label{conjugate-rode}
\begin{split}
\frac{dx^{(j)}}{dt}& =F^{(j)}(x^{(j)},O_t^{(j)}(\omega )) \\
& :=e^{-O_t^{(j)}(\omega )}f^{(j)}(e^{O_t^{(j)}(\omega
)}x^{(j)})+O_t^{(j)}(\omega )x^{(j)},\quad j=1,\dots ,N
\end{split}
\end{equation}
(see \cite{18} for the conjugate theory of SODE and RODE).

Now we consider the linear coupled RODEs of \eqref{conjugate-rode}
\begin{equation}\label{linear-coupled-rode}
\begin{split}
\frac{dx^{(j)}}{dt}=F^{(j)}(x^{(j)},O_t^{(j)}(\omega ))+\nu \big(%
x^{(j-1)}-2x^{(j)}+x^{(j+1)}\big),\quad j=1,\dots ,N
\end{split}
\end{equation}
with coupling coefficient $\nu >0$, where $%
x^{(0)}=x^{(N)}$ and $x^{(N+1)}=x^{(1)}$.  Now
\eqref{linear-coupled-rode} can be written as the following
equivalent SODEs
\begin{equation}\label{equivalent-sode}
\begin{split}
dX_t^{(j)} &=\bigg(f^{(j)}(X_t^{(j)})+\nu \big(e^{\rho
_t^{(j)}}X_t^{(j-1)}-2X_t^{(j)}+e^{\varrho
_t^{(j)}}X_t^{(j+1)}\big)\bigg)dt\\
&\quad +\sum\limits_{i=1}^mc_i^{(j)}X_t^{(j)}\circ dW_t^{(i)},\quad
j=1,\dots ,N,
\end{split}
\end{equation}
where $\rho _t^{(j)}=O_t^{(j)}-O_t^{(j-1)},\varrho
_t^{(j)}=O_t^{(j)}-O_t^{(j+1)},O_t^{(0)}=O_t^{(N)}$ and $%
O_t^{(N+1)}=O_t^{(1)}$.

For synchronization of solutions to coupled RODEs
\eqref{linear-coupled-rode}, there are two cases: one for any two
solutions and the other for components of solutions. When $N=2$,
i.e. for two Stratonovich SODEs, Caraballo, Kloeden \& Neuenkirch
\cite{9} considered both types of synchronization. Under the
assumption of one-sided dissipative Lipschitz conditions
\eqref{dissipative-condition}, they first proved that
synchronization of any two solutions occurs and the random dynamical
system generated by the solution of
\eqref{linear-coupled-rode}$_{N=2}$ has a singleton sets random
attractor; then they proved that the synchronization between any two
components of solutions occurs as the coupled coefficient $\nu$
tends to infinity. Moreover, when the driving noise is same in each
system, exact synchronization occurs no matter how large the
intensity coefficients of noise are. Based on the work of \cite{9},
in this paper we consider the above two types of synchronization of
solutions of \eqref{linear-coupled-rode} in the case of $N\geq 3$
and obtain similar results. Explicitly, we show that the random
dynamical system generated by the solution of the coupled RODEs
\eqref{linear-coupled-rode} has a singleton sets random attractor
which implies the synchronization of any two solutions of
\eqref{linear-coupled-rode}. Moreover, the singleton sets random
attractor determines a stationary stochastic solution of the
equivalently coupled SODEs \eqref{equivalent-sode}.  We also show
that any solutions of RODEs \eqref{linear-coupled-rode} converge to
a solution $\bar z(t,\omega )$ of the averaged RODE
\begin{equation}\label{average-rode}
\begin{split}
\frac{dz}{dt}=\frac
1N\sum\limits_{j=1}^Ne^{-O_t^{(j)}}f^{(j)}(e^{O_t^{(j)}}z)+\frac
1N\sum\limits_{j=1}^NO_t^{(j)}z
\end{split}
\end{equation}
as the coupling coefficient $\nu\rightarrow\infty$. Here it is worth
to mention that this generalization is not trivial since some new
techniques are used especially in section \ref{syn-comp-solutions}.

The rest of this paper is organized as follows. In section
\ref{two-lemmas}, we introduce two lemmas which will be used
frequently. In section \ref{syn-two-solutions}, we show
synchronization of two solutions to the coupled RODEs and obtain the
stationary stochastic solution to the equivalent SODEs. In section
\ref{syn-comp-solutions}, we study synchronization of components of
solutions to the coupled RODEs and obtain the exact synchronization
of the equivalent SODEs provided by same driving noise.

%-----------------------------------------------------------------------------------------------------

\section{Two Lemmas}\label{two-lemmas}

We will frequently use the following two lemmas.

\begin{lemma}\label{property-OU}
There exists a $(\theta_t)_{t\in \mathbb{R}}$ invariant subset
$\overline{\Omega}\in\mathcal{F}$ of
$\Omega=C_0(\mathbb{R},\mathbb{R})$ of full measure such that for
$\omega\in\overline{\Omega}$,
\begin{equation}\label{sublinear-growth}
\begin{split}
\lim\limits_{t\rightarrow\pm\infty}\frac{|\omega(t)|}t=0,
\end{split}
\end{equation}
and for $j=1,...,N,$ there exist random variables
$\overline{O}^{(j)}=O_t^{(j)}$ and $T_{\omega}>0$ such that
\begin{equation}\label{ergodic-property}
\begin{split}
\overline{O}^{(j)}(\theta _t\omega)=O_t^{(j)}(\omega),\quad
\lim\limits_{t\rightarrow\pm\infty}\frac
1t\int_0^t\overline{O}^{(j)}(\theta_{\tau}\omega)d\tau=0,\quad
\omega\in\overline{\Omega},
\end{split}
\end{equation}
and
\[
\begin{split}
e^{2\int_s^tO_\tau ^{(j)}d\tau }\leq e^{\frac L2(t-s)}\quad for%
-s,t>T_\omega .
\end{split}
\]
\end{lemma}

\begin{proof}
The equalities \eqref{sublinear-growth}-\eqref{ergodic-property} can
be found in \cite{9,10}. By \eqref{ergodic-property},
$\lim\limits_{t\rightarrow\infty}\frac{1}{t}\int_0^tO_\tau^{(j)}d\tau=0$,
thus, there exists $T_{\omega}(1)>0$ such that
$\int_0^tO_\tau^{(j)}d\tau\leq\frac{L}{4}t$ for $t>T_{\omega}(1)$.
Similarly,
$\lim\limits_{s\rightarrow-\infty}\frac{1}{s}\int_s^0O_\tau^{(j)}d\tau=0$
implies that there exists $T_{\omega}(2)>0$ such that
$\int_s^0O_\tau^{(j)}d\tau\leq-\frac{L}{4}s$ for $-s>T_{\omega}(2)$.
Taking $T_{\omega}=\max\{T_{\omega}(1),T_{\omega}(2)\}$, we have
$2\int_s^tO_\tau^{(j)}d\tau\leq\frac{L}{2}(t-s)$ for
$-s,t>T_{\omega}$, which yields the assertion.
\end{proof}

We remark that the proof of \eqref{sublinear-growth} and
\eqref{ergodic-property} requires the ergodicity of the metric
dynamical system
$(\Omega,\mathcal{F},\mathbb{P},(\theta_t)_{t\in\mathbb{R}})$. In
the following sections, since $\overline{\Omega}$ is an
$(\theta_t)_{t\in\mathbb{R}}$ invariant set with full measure, we
consider $(\theta_t)_{t\in\mathbb{R}}$ defined on
$\overline{\Omega}$ instead of $\Omega$. This mapping has the same
properties as the original one if we choose for $\mathcal{F}$ the
trace $\sigma$-algebra with respect to $\overline{\Omega}$.

\begin{lemma}\label{diff-inte-ineq}
Suppose that $A(t)$ is a $p\times p$ matrix and $\varphi (t),\,\psi
(t)$ are $p$-dimensional vectors on $[t_0,t](t\geq t_0,\,t,t_0\in
\mathbb{R})$ which are sufficiently regular. If the following
inequality holds in the componentwise sense
\begin{equation}\label{differential-inequality}
\frac d{dt}\varphi (t)\leq A(t)\varphi (t)+\psi (t),\quad t\geq t_0,
\end{equation}
then
\begin{equation}\label{integral-inequality}
\begin{split}
\varphi (t)\leq \exp \Bigg(\int_{t_0}^tA(\tau )d\tau \Bigg)\varphi
(t_0)+\int_{t_0}^t\exp \Bigg(\int_u^tA(\tau )d\tau \Bigg)\psi (u)du,\quad
t\geq t_0.
\end{split}
\end{equation}
\end{lemma}

\begin{proof}
It follows from \eqref{differential-inequality} that
\[
\begin{split}
\frac{d}{dt}\Bigg(\exp\Bigg(-\int_{t_0}^tA(\tau)d\tau\Bigg)\varphi(t)\Bigg)
&=\exp\Bigg(-\int_{t_0}^tA(\tau)d\tau\Bigg)\big(\frac{d}{dt}\varphi(t)-A(t)\varphi(t)\big)\\
&\leq\exp\Bigg(-\int_{t_0}^tA(\tau)d\tau\Bigg)\psi(t),\\
\end{split}
\]
then
\[
\begin{split}
\exp\Bigg(-\int_{t_0}^tA(\tau)d\tau\Bigg)\varphi(t)-\varphi(t_0)\leq
\int_{t_0}^t\exp\Bigg(-\int_{t_0}^uA(\tau)d\tau\Bigg)\psi(u)du,
\end{split}
\]
which implies inequality \eqref{integral-inequality}.
\end{proof}
%--------------------------------------------------------------------------------------------------

\section{Synchronization of Two Solutions}\label{syn-two-solutions}

Consider the coupled RODEs \eqref{linear-coupled-rode}
\begin{equation}\label{linear-coupled-rode-again}
\begin{split}
\frac{dx^{(j)}}{dt}=F^{(j)}(x^{(j)},O_t^{(j)}(\omega ))+\nu \big(%
x^{(j-1)}-2x^{(j)}+x^{(j+1)}\big),\quad j=1,\dots ,N
\end{split}
\end{equation}
with initial data
\begin{equation}\label{initial-data}
\begin{split}
x^{(j)}(0,\omega )=x_0^{(j)}(\omega )\in \mathbb{R}^d,\quad
\omega\in\Omega ,\quad j=1,\dots ,N,
\end{split}
\end{equation}
where $\nu >0,$ and
\begin{equation}\label{brief-notation-F}
\begin{split}
F^{(j)}(x^{(j)},O_t^{(j)}(\omega ))=e^{-O_t^{(j)}(\omega
)}f^{(j)}(e^{O_t^{(j)}(\omega )}x^{(j)})+O_t^{(j)}(\omega
)x^{(j)},\quad j=1,\dots ,N.
\end{split}
\end{equation}
Here $f^{(j)}$ are regular enough to ensure the existence and
uniqueness of global solutions on $\mathbb{R}$ and satisfy the
one-sided dissipative Lipschitz conditions
\eqref{dissipative-condition} for $j=1,...,N$.

For asymptotic behavior of the difference between two solutions of
RODEs \eqref{linear-coupled-rode-again}-\eqref{initial-data}, we
have

\begin{lemma}\label{two-solution-converge}
For any two solutions $\big(x_1^{(1)}(t),x_1^{(2)}(t),\dots
,x_1^{(N)}(t)\big)^{\top}$ and $\big(x_2^{(1)}(t),x_2^{(2)}(t),\dots
,x_2^{(N)}(t)\big)^{\top}$ of RODEs
\eqref{linear-coupled-rode-again}-\eqref{initial-data} (omitting
$\omega$ for brevity),
\[
\lim\limits_{t\rightarrow \infty
}\|x_1^{(j)}(t)-x_2^{(j)}(t)\|=0,\quad j=1,\dots ,N,
\]
that is, all solutions of the coupled RODEs
\eqref{linear-coupled-rode-again}-\eqref{initial-data} converge
pathwise to each other as time $t$ goes to infinity.
\end{lemma}

\begin{proof}
By the one-sided dissipative Lipschitz conditions \eqref{dissipative-condition}, we obtain for $j=1,\dots,N$,
\[
\begin{split}
\frac{d}{dt}\|x_1^{(j)}(t)-x_2^{(j)}(t)\|^2
&=2\Big\langle %
x_1^{(j)}(t)-x_2^{(j)}(t),\frac{d}{dt}x_1^{(j)}(t)-\frac{d}{dt}x_2^{(j)}(t)%
\Big\rangle\\
&=2e^{-O_t^{(j)}}\Big\langle %
x_1^{(j)}(t)-x_2^{(j)}(t),f^{(j)}(e^{O_t^{(j)}}x_1^{(j)}(t))-f^{(j)}(e^{O_t^{(j)}}x_2^{(j)}(t))%
\Big\rangle\\
& \quad
+\big(2O_t^{(j)}-4\nu\big)\|x_1^{(j)}(t)-x_2^{(j)}(t)\|^2\\
& \quad +2\nu\Big\langle %
x_1^{(j)}(t)-x_2^{(j)}(t),x_1^{(j-1)}(t)-x_2^{(j-1)}(t)+x_1^{(j+1)}(t)-x_2^{(j+1)}(t)%
\Big\rangle\\
&\leq\big(2O_t^{(j)}-2L-2\nu\big)\|x_1^{(j)}(t)-x_2^{(j)}(t)\|^2\\
& \quad
+\nu\|x_1^{(j-1)}(t)-x_2^{(j-1)}(t)\|^2+\nu\|x_1^{(j+1)}(t)-x_2^{(j+1)}(t)\|^2.
\end{split}
\]
 Define
\[
\begin{split}
\mathbf{x}(t)=\Big(\|x_1^{(1)}(t)-x_2^{(1)}(t)\|^2,
\|x_1^{(2)}(t)-x_2^{(2)}(t)\|^2
,\dots,\|x_1^{(N)}(t)-x_2^{(N)}(t)\|^2\Big)^{\top}, \
t\in\mathbb{R},
\end{split}
\]
and
\[
\begin{split}
A_\nu(t)=
\begin{pmatrix}
a_\nu^{(1)}(t)    &   \nu             &  0    &  \cdots  & 0 &  \nu\\
\nu               &  a_\nu^{(2)}(t)   &  \nu   & 0  & \cdots  &  0\\
 0        &   \nu       &   a_\nu^{(3)}(t)   & \ddots &  \ddots  &   \vdots \\
\vdots     & \ddots     & \ddots   &   \ddots  &  \nu  &   0\\
0   & \cdots  & 0  &  \nu &  a_\nu^{(N-1)}(t) &   \nu\\
\nu   &  0 & \cdots  & 0  &  \nu & a_\nu^{(N)}(t)
\end{pmatrix}, \ t\in\mathbb{R},
\end{split}
\]
where diagonal entries $a_\nu^{(j)}(t)=2O_t^{(j)}-2L-2\nu,\
j=1,\dots,N$. Thus  the above  differential inequalities can be written as
a simple form
\begin{equation}\label{diff-ineq-simple-form}
\begin{split}
\dot{\mathbf{x}}(t)\leq A_\nu(t)\mathbf{x}(t),
\end{split}
\end{equation}
componentwisely. By Lemma \ref{diff-inte-ineq} and
\eqref{diff-ineq-simple-form} we obtain
\begin{equation*}
\begin{split}
\mathbf{x}(t)\leq
\exp\Bigg(\int_0^tA_\nu(\tau)d\tau\Bigg)\mathbf{x}(0)
\end{split}
\end{equation*}
componentwisely.

The proof of this lemma will be completed in the following  Lemma
\ref{estimation-expon-1}.
\end{proof}

\begin{lemma}\label{estimation-expon-1}
For $t\geq T_\omega $ and $\nu >0$,
\begin{equation*}
\begin{split}
\big\|\exp\Bigg(\int_0^tA_\nu (\tau )d\tau
\Bigg)\mathbf{x}(0)\big\|\leq e^{-Lt}\|\mathbf{x}(0)\|,
\end{split}
\end{equation*}
where $T_\omega $ is defined as in Lemma \ref{property-OU}.

\begin{proof}
Matrix $\int_0^tA_\nu(\tau)d\tau$ is real symmetric, which implies that
there exists an orthonormal basis consisting of eigenvectors
$\eta_{\nu,t}^{(1)},\eta_{\nu,t}^{(2)},\dots,\eta_{\nu,t}^{(N)}$ of
$\mathbb{R}^{N}$ with eigenvalues
$\lambda_{\nu,t}^{(1)},\lambda_{\nu,t}^{(2)},\dots,\lambda_{\nu,t}^{(N)}$,
and therefore there exist
$c_{\mathbf{x}(0),\nu,t}^{(1)},c_{\mathbf{x}(0),\nu,t}^{(2)},\dots,c_{\mathbf{x}(0),\nu,t}^{(N)}$
such that
\[
\begin{split}
\mathbf{x}(0)=\sum\limits_{j=1}^Nc_{\mathbf{x},\nu,t}^{(j)}\eta_{\nu,t}^{(j)}.
\end{split}
\]

Since
$\eta_{\nu,t}^{(1)},\eta_{\nu,t}^{(2)},\dots,\eta_{\nu,t}^{(N)}$ are
orthogonal and
$\exp\Big(\int_0^tA_\nu(\tau)d\tau\Big)\eta_{\nu,t}^{(j)}=e^{\lambda_{\nu,t}^{(j)}}\eta_{\nu,t}^{(j)}$ for $
j=1,\dots,N$, we have
\begin{equation}\label{estimation-expon-2}
\begin{split}
\big\|\exp\Bigg(\int_0^tA_\nu(\tau)d\tau\Bigg)\mathbf{x}(0)\big\|^2
&=\big\|\sum\limits_{j=1}^Nc_{\mathbf{x}(0),\nu,t}^{(j)}\exp\Bigg(\int_0^tA_\nu(\tau)d\tau\Bigg)\eta_{\nu,t}^{(j)}\big\|^2\\
&=\big\|\sum\limits_{j=1}^Ne^{\lambda_{\nu,t}^{(j)}}c_{\mathbf{x}(0),\nu,t}^{(j)}\eta_{\nu,t}^{(j)}\big\|^2\\
&\leq
e^{2\max\{\lambda_{\nu,t}^{(1)},\lambda_{\nu,t}^{(2)},\dots,\lambda_{\nu,t}^{(N)}\}}\|\mathbf{x}(0)\|^2.
\end{split}
\end{equation}

Next, let us estimate the upper bound of eigenvalues of matrix
$\int_0^tA_\nu(\tau)d\tau$. The quadratic form satisfies
\[
\begin{split}
f(\xi_1,\xi_2,\dots,\xi_N) &=\xi^{\top}\Bigg(\int_0^tA_\nu(\tau)d\tau\Bigg)\xi\\
 &=\sum\limits_{j=1}^N\Bigg(2\int_0^tO_\tau^{(j)}d\tau-2Lt-2\nu{t}\Bigg)\xi_j^2+2\nu{t}\sum\limits_{j=1}^{N}\xi_j\xi_{j-1}\\
 &\leq\sum\limits_{j=1}^N\Bigg(2\int_0^tO_\tau^{(j)}d\tau-Lt\Bigg)\xi_j^2-Lt\sum\limits_{j=1}^N\xi_j^2\\
\end{split}
\]
where $\xi=(\xi_1,\xi_2,\dots,\xi_N)^{\top}\in\mathbb{R}^N$, and
$\xi_0=\xi_N$. Hence, it follows from Lemma \ref{property-OU} that
\begin{equation}\label{estimation-quadratic-form}
\begin{split}
f(\xi_1,\xi_2,\dots,\xi_N)\leq-Lt\sum\limits_{j=1}^N\xi_j^2,
\end{split}
\end{equation}
for $t\geq T_\omega$ and for all $\nu>0$.  Inequality
\eqref{estimation-quadratic-form} implies that the quadratic form is
negative definite and eigenvalues of $\int_0^tA_\nu(\tau)d\tau$
satisfy
\begin{equation}\label{estimation-eigenvalues}
\begin{split}
\max\Big\{\lambda_{\nu,t}^{(1)},\lambda_{\nu,t}^{(2)},\dots,\lambda_{\nu,t}^{(N)}\Big\}\leq-Lt.
\end{split}
\end{equation}
Combining \eqref{estimation-expon-2} and
\eqref{estimation-eigenvalues} yields the assertion.
\end{proof}
\end{lemma}

Now we use the theory of random dynamical systems to find what the
solutions of \eqref{linear-coupled-rode-again}-\eqref{initial-data}
will converge to. It is easy to see from \cite{6} that the solution
\[
\phi
(t,\omega )=\big(x^{(1)}(t,\omega ),x^{(2)}(t,\omega ),...,x^{(N)}(t,\omega )%
\big)^{\top},\quad\omega \in \Omega
\]
of \eqref{linear-coupled-rode-again}-\eqref{initial-data} generates
a random dynamical system over
$(\Omega,\mathcal{F},\mathbb{P},(\theta_t)_{t\in\mathbb{R}})$ with
state space $\mathbb{R}^{Nd}.$ For this random dynamical system
$\phi(t,\omega ),$ we have
\begin{theorem}
$\phi (t,\omega ),$ $t\in \mathbb{R},$ $\omega \in \Omega ,$ has a
singleton sets random attractor $\big\{A_\nu (\omega )\big\}$ where
\[
\begin{split}
A_\nu (\omega )=\big(\bar x_\nu ^{(1)}(\omega ),\bar x_\nu
^{(2)}(\omega ),\dots ,\bar x_\nu ^{(N)}(\omega )\big)^{\top},
\end{split}
\]
which implies the synchronization of any two solutions of
\eqref{linear-coupled-rode-again}-\eqref{initial-data}. Moreover,
\[
\begin{split}
\Big(\bar x_\nu ^{(1)}(\theta _t\omega )e^{O_t^{(1)}(\omega )},\bar
x_\nu ^{(2)}(\theta _t\omega )e^{O_t^{(2)}(\omega )},\dots ,\bar
x_\nu ^{(N)}(\theta _t\omega )e^{O_t^{(N)}(\omega )}\Big)^{\top}
\end{split}
\]
is the stationary stochastic solution of the equivalently coupled
SODEs \eqref{equivalent-sode}.
\end{theorem}

\begin{proof}  First,
\[
\begin{split}
\frac{d}{dt}\|x^{(j)}(t)\|^2 &= 2\Big\langle x^{(j)}(t),\frac{d}{dt}%
x^{(j)}(t)\Big\rangle \\
&= 2\Big\langle x^{(j)}(t),e^{-O_t^{(j)}}f^{(j)}(e^{O_t^{(j)}}x^{(j)}(t))%
\Big\rangle+2\Big\langle x^{(j)}(t),O_t^{(j)}x^{(j)}(t)\Big\rangle \\
& \quad +2\nu\Big\langle x^{(j)}(t),x^{(j-1)}(t)-2x^{(j)}(t)+x^{(j+1)}(t)%
\Big\rangle \\
&\leq 2e^{-2O_t^{(j)}}\Big\langle %
e^{O_t^{(j)}}x^{(j)}(t)-0,f^{(j)}(e^{O_t^{(j)}}x^{(j)}(t))-f^{(j)}(0)%
\Big\rangle \\
& \quad +2e^{-O_t^{(j)}}\Big\langle x^{(j)}(t),f^{(j)}(0)\Big\rangle+\big(%
2O_t^{(j)}-4\nu\big)\|x^{(j)}(t)\|^2 \\
& \quad +2\nu\Big\langle x^{(j)}(t),x^{(j-1)}(t)+x^{(j+1)}(t)\Big\rangle \\
&\leq\big(2O_t^{(j)}-2L-2\nu\big)\|x^{(j)}(t)\|^2+\nu\|x^{(j-1)}(t)\|^2+\nu\|x^{(j+1)}(t)\|^2 \\
& \quad+2\|x^{(j)}(t)\|\|f^{(j)}(0)\|e^{-O_t^{(j)}} \\
&\leq\big(2O_t^{(j)}-L-2\nu\big)\|x^{(j)}(t)\|^2+\nu\|x^{(j-1)}(t)\|^2+\nu\|x^{(j+1)}(t)\|^2 \\
& \quad +\frac{e^{-2O_t^{(j)}}}{L}\|f^{(j)}(0)\|^2,
\end{split}
\]
for $j=1,\dots,N$. Analogous to \eqref{diff-ineq-simple-form}, we
obtain
\[
\begin{split}
\dot{\widetilde{\mathbf{x}}}(t)\leq \widetilde{A}_\nu(t)\widetilde{\mathbf{x}%
}(t)+\widetilde{\mathbf{f}}(t)
\end{split}
\]
with
\[
\begin{split}
\widetilde{\mathbf{x}}(t)=\Big(\|x^{(1)}(t)\|^2,\|x^{(2)}(t)\|^2,\dots,\|x^{(N)}(t)\|^2\Big)^{\top},\quad
t\in\mathbb{R},
\end{split}
\]
\[
\begin{split}
\widetilde{\mathbf{f}}(t)=\frac{1}{L}\Big(e^{-2O_t^{(1)}}\|f^{(1)}(0)\|^2,e^{-2O_t^{(2)}}\|f^{(2)}(0)\|^2,\dots,e^{-2O_t^{(N)}}\|f^{(N)}(0)\|^2\Big)^{\top},\quad
t\in\mathbb{R}
\end{split}
\]
and
\[
\begin{split}
\widetilde{A}_\nu(t)=
\begin{pmatrix}
\widetilde{a}_\nu^{(1)}(t)    &   \nu             &  0    &  \cdots  & 0 &  \nu\\
\nu               &  \widetilde{a}_\nu^{(2)}(t)   &  \nu   & 0  & \cdots  &  0\\
 0        &   \nu       &   \widetilde{a}_\nu^{(3)}(t)   & \ddots &  \ddots  &   \vdots \\
\vdots     & \ddots     & \ddots   &   \ddots  &  \nu  &   0\\
0   & \cdots  & 0  &  \nu &  \widetilde{a}_\nu^{(N-1)}(t) &   \nu\\
\nu   &  0 & \cdots  & 0  &  \nu & \widetilde{a}_\nu^{(N)}(t)
\end{pmatrix}, \ t\in\mathbb{R},
\end{split}
\]
where diagonal entries
$\widetilde{a}_\nu^{(j)}(t)=2O_t^{(j)}-L-2\nu$ for $\ j=1,\dots,N$.
Then by Lemma \ref{diff-inte-ineq},
\[
\begin{split}
\widetilde{\mathbf{x}}(t)\leq \exp\Bigg(\int_{t_0}^t\widetilde{A}%
_\nu(\tau)d\tau\Bigg)\widetilde{\mathbf{x}}(t_0)+\int_{t_0}^t\exp\Bigg(%
\int_u^t\widetilde{A}_\nu(\tau)d\tau\Bigg)\widetilde{\mathbf{f}}(u)du,\quad
t\geq t_0.
\end{split}
\]
Analogous to Lemma \ref{estimation-expon-1}, we have
\[
\begin{split}
\big\|\exp\Bigg(\int_{t_0}^t\widetilde{A}_\nu (\tau )d\tau \Bigg)%
\widetilde{\mathbf{x}}(t_0)\big\|\leq e^{-\frac L2(t-t_0)}\|\widetilde{%
\mathbf{x}}(t_0)\|,\quad -t_0\,,t\geq T_\omega, \quad \nu>0.
\end{split}
\]
Define
\begin{equation}\label{notation-1}
\begin{split}
C_\nu(\omega):=\int_{-\infty}^0\exp\Bigg(\int_u^0\widetilde{A}_\nu(\tau)d\tau%
\Bigg)\widetilde{\mathbf{f}}(t)du,
\end{split}
\end{equation}
\[
\begin{split}
R_\nu^2(\omega)=1+\|C_\nu(\omega)\|^2
\end{split}
\]
and let $B_\nu(\omega)$ be a random ball in $\mathbb{R}^{Nd}$
centered at the origin with radius $R_\nu(\omega)$. Note that the
infinite integral on the right-hand side of \eqref{notation-1} is
well defined by Lemma \ref{property-OU}.

Note that if  $\lim\limits_{t\rightarrow\infty}
e^{-kt}\|\widetilde{\mathbf{x}}(t_0)\|=0$ for all $k>0$, then
\[
\begin{split}
\sum\limits_{j=1}^N\|x^{(j)}(0)\|^2<R_\nu^2(\omega)\quad \text{as}%
\quad t_0\rightarrow-\infty,
\end{split}
\]
which implies that the closed random ball $B_\nu(\omega)$ is a
pullback absorbing set at $t=0$ of $\phi (t,\omega )$, that is, for
any $\omega \in \Omega $ and any $D\in \mathcal{D}$ ($\mathcal{D}$
is a collection of tempered random bounded sets, i.e.
$\lim\limits_{t\rightarrow \infty}e^{-kt}\sup_{u\in D(\theta
_{-t}\omega )}\|u\|=0$), there exists $t_{B_\nu}(\omega )$ such that
\[
\begin{split}
\phi (t,\theta _{-t}\omega )D(\theta _{-t}\omega )\subset
B_\nu(\omega )\quad \text{for all}\quad t\geq t_{B_\nu}(\omega ).
\end{split}
\]
Hence by Theorem 4.1 in \cite{14}, the random
dynamical system $\phi (t,\omega )$ generated by the coupled RODEs \eqref{linear-coupled-rode-again}-\eqref{initial-data} has a random attractor $%
A_\nu(\omega)$ in $B_\nu(\omega)$ for each $\omega$ with the
properties that $A_\nu(\omega)$ is compact, $\phi$-invariant
($\phi(t,\omega)A_\nu(\omega)=A_\nu(\theta_t\omega)$ for all
$t\geq0$ and $\omega\in\Omega$) and attracting in $\mathcal{D}$,
i.e. for all $D\in\mathcal{D}$,
\[
\begin{split}
\lim\limits_{t\rightarrow
+\infty}H^{*}_d\big(\phi(t,\theta_{-t}\omega)D(\theta_{-t}\omega),A_\nu(\omega)\big)=0,\quad
\omega\in\Omega,
\end{split}
\]
where $H^{*}_d$ is the Hausdorff semi-distance on $\mathbb{R}^{Nd}$.
By Lemma \ref{two-solution-converge}, all solutions of
\eqref{linear-coupled-rode-again}-\eqref{initial-data} converge
pathwise to each other, therefore, $A_\nu(\omega)$ consists of
singleton sets, i.e.
\[
\begin{split}
A_\nu(\omega)=\Big(\bar{x}^{(1)}_\nu(\omega),\bar{x}^{(2)}_\nu(\omega),\dots,%
\bar{x}^{(N)}_\nu(\omega)\Big)^{\top}.
\end{split}
\]

When we transfer the coupled RODEs \eqref{linear-coupled-rode-again}
back to the coupled SODEs \eqref{equivalent-sode}, the corresponding
pathwise singleton sets attractor is then
\[
\begin{split}
\Big(\bar{x}^{(1)}_\nu(\theta_t\omega)e^{O_t^{(1)}(\omega)},\bar{x}%
^{(2)}_\nu(\theta_t\omega)e^{O_t^{(2)}(\omega)},\dots,\bar{x}%
^{(N)}_\nu(\theta_t\omega)e^{O_t^{(N)}(\omega)}\Big)^{\top},
\end{split}
\]
which is exactly a stationary stochastic solution of the coupled
SODEs \eqref{equivalent-sode} because the Ornstein-Uhlenbeck process
is stationary.
\end{proof}

\section{Synchronization of Components of
Solutions}\label{syn-comp-solutions}

It is known in section \ref{two-lemmas} that all solutions of the
coupled RODEs \eqref{linear-coupled-rode-again}-\eqref{initial-data}
converge pathwise to each other in the future for a fixed $\nu
(>0)$. Here, we consider what will happen to solutions of the
coupled RODEs \eqref{linear-coupled-rode-again}-\eqref{initial-data}
as the coupling coefficient $\nu$ goes to infinity. First, we prove
a lemma which plays an important role in this section.

\begin{lemma}\label{key-lemma}
For fixed $p\in \mathbb{N}$ and any $\alpha \in (0,2)$, there exist
a $\alpha _0(p)\in (0,2)$ such that the $p\times p$ real symmetric
triple diagonal matrix
\[
\begin{split}
A=%
\begin{pmatrix}
-\alpha    &   1             &  0    &  \cdots  & 0 \\
1     &  -\alpha   &  1     & \ddots  &  \vdots\\
 0        &  1       &  \ddots &  \ddots  &   0 \\
\vdots     & \ddots     & \ddots   &  -\alpha  & 1\\
0   & \cdots  & 0  & 1  &  -\alpha\end{pmatrix}
\end{split}
\]
is negative definite for all $\alpha \geq \alpha _0(p)$.
\end{lemma}

\begin{proof}
Let $\widetilde{A}=-A$. We assert that there exists an
$\alpha_0(p)\in(0,2)$ such that $\widetilde{A}$ is positive definite
for $\alpha\geq\alpha_0(p)$,then $A$ is negative definite for
$\alpha\geq\alpha_0(p)$. In fact, let $a+b=\lambda-\alpha$ and $ab=1$,
we have
\[
\begin{split}
\big|\lambda E-\widetilde{A}\big|&=\begin{vmatrix}
\lambda-\alpha    &   1             &  0    &  \cdots  & 0 \\
1     &  \lambda-\alpha   &  1     & \ddots  &  \vdots\\
 0        &  1       &  \ddots &  \ddots  &   0 \\
\vdots     & \ddots     & \ddots   &  \lambda-\alpha  & 1\\
0   & \cdots  & 0  & 1  &  \lambda-\alpha\end{vmatrix}\\
&=\begin{vmatrix}
a+b    &   ab             &  0    &  \cdots  & 0 \\
1     &  a+b   &  ab     & \ddots  &  \vdots\\
 0        &  1       &  \ddots &  \ddots  &   0 \\
\vdots     & \ddots     & \ddots   &  a+b  & ab\\
0   & \cdots  & 0  & 1  & a+b\end{vmatrix}
=\frac{a^{p+1}-b^{p+1}}{a-b}.\\
\end{split}
\]

If $a\neq b$,  then $\big|\lambda E-\widetilde{A}\big|=0$ implies that
$(\frac{a}{b})^{p+1}=1$ and thus $\frac{a}{b}=e^{i\frac{2k\pi}{p+1}}$ for
$k=1,\dots,p$ ($k\neq0$  since $\frac{a}{b}\neq1$). It follows from
$ab=1$ that
\[
\begin{split}
a_k=\pm\Big(\cos\frac{k\pi}{p+1}+i\sin\frac{k\pi}{p+1}\Big),\quad
b_k=\pm\Big(\cos\frac{k\pi}{p+1}-i\sin\frac{k\pi}{p+1}\Big),
\end{split}
\]
$k=1,\dots,p$. Then
$\lambda=\alpha+(a+b)=\alpha\pm2\cos\frac{k\pi}{p+1},\ k=1,\dots,p$.
Note that $\cos\frac{k\pi}{p+1}=-\cos\frac{(p+1-k)\pi}{p+1}$, thus the
$p$ different eigenvalues of $\widetilde{A}$ are
$\lambda_k=\alpha+2\cos\frac{k\pi}{p+1},\ k=1,\dots,p$.

Otherwise, $a=b$ implies $a=b=1$ or $a=b=-1$, then $\lambda=\alpha+2$ or
$\lambda=\alpha-2$. But $\big|\lambda E-\widetilde{A}\big|\neq0$ for
these two $\lambda$. Hence, all eigenvalues of $\widetilde{A}$ are
\[
\begin{split}
\lambda_k=\alpha+2\cos\frac{k\pi}{p+1},\ k=1,\dots,p.
\end{split}
\]
It follows that for any
$\alpha_0(p)\in\big(-2\cos\frac{p\pi}{p+1},2\big)\subset(0,2)$, for
example, $\alpha_0(p)=1-\cos\frac{p\pi}{p+1}$, $\widetilde{A}$ is
positive definite for $\alpha\geq\alpha_0(p)$.
\end{proof}

We also need the following estimations. Suppose that $\phi
(t)=\big(x_\nu ^{(1)}(t),x_\nu ^{(2)}(t),\dots ,x_\nu
^{(N)}(t)\big)^{\top}$ is a solution of the coupled RODEs
\eqref{linear-coupled-rode-again}-\eqref{initial-data}. For any two
different components $x_\nu ^{(k)}(t)$, $x_\nu ^{(l)}(t)$ of the
solution,
\[
\begin{split}
y_\nu ^{k,l}(t)&=2\Big\langle x_\nu ^{(k)}(t)-\ x_\nu
^{(l)}(t),F^{(k)}(x_\nu ^{(k)}(t),O_t^{(k)})-F^{(l)}(x_\nu
^{(l)}(t),O_t^{(l)})\Big\rangle \\
& =2\Big\langle x_\nu
^{(k)}(t)-x_\nu
^{(l)}(t),e^{-O_t^{(k)}}f^{(k)}(e^{O_t^{(k)}}x^{(k)}(t))-e^{-O_t^{(l)}}f^{(l)}(e^{O_t^{(l)}}x^{(l)}(t))%
\Big\rangle \\
& \quad +2\Big\langle x_\nu ^{(k)}(t)-x_\nu
^{(l)}(t),O_t^{(k)}x^{(k)}(t)-O_t^{(l)}x^{(l)}(t)\Big\rangle \\
& \leq 2\|x_\nu ^{(k)}(t)-x_\nu ^{(l)}(t)\|\Big(e^{-O_t^{(k)}}\|%
f^{(k)}(e^{O_t^{(k)}}x^{(k)}(t))\|+|O_t^{(k)}|\cdot\|x^{(k)}(t)\|\Big) \\
& \quad +2\|x_\nu ^{(k)}(t)-x_\nu ^{(l)}(t)\|\Big(e^{-O_t^{(l)}}%
\|f^{(l)}(e^{O_t^{(l)}}x^{(l)}(t))\|+|O_t^{(l)}|\cdot\|x^{(l)}(t)\|\Big),
\end{split}
\]
thus, for fixed $\beta >0$, we have
\[
\begin{split}
-\beta \nu \|x_\nu ^{(k)}(t)-x_\nu ^{(l)}(t)\|^2+y_\nu ^{k,l}(t)&
\leq \frac 1\nu \Bigg(\frac 4\beta e^{-2O_t^{(k)}}\|%
f^{(k)}(e^{O_t^{(k)}}x^{(k)}(t))\|^2+\frac 4\beta |O_t^{(k)}|^2\|x^{(k)}(t)\|^2\Bigg)\\
& \quad +\frac 1\nu \Bigg(\frac 4\beta e^{-2O_t^{(l)}}\|%
f^{(l)}(e^{O_t^{(l)}}x^{(l)}(t))\|^2+\frac 4\beta
|O_t^{(l)}|^2\|x^{(l)}(t)\|^2\Bigg).
\end{split}
\]
Let
\[
\begin{split}
M_{T_1,T_2}^{k,l,\beta }(\nu ,\omega )& =\sup\limits_{t\in [T_1,T_2]}\Bigg(%
\frac 4\beta e^{-2O_t^{(k)}}\|f^{(k)}(e^{O_t^{(k)}}x^{(k)}(t))\|%
^2+\frac 4\beta |O_t^{(k)}|^2\|x^{(k)}(t)\|^2\Bigg) \\
& \quad +\sup\limits_{t\in [T_1,T_2]}\Bigg(\frac 4\beta e^{-2O_t^{(l)}}\|%
f^{(l)}(e^{O_t^{(l)}}x^{(l)}(t))\|^2+\frac 4\beta
|O_t^{(l)}|^2\|x^{(l)}(t)\|^2\Bigg)
\end{split}
\]
for any bounded interval $[T_1,T_2]$. Note that $C_\nu (\omega )$ in
\eqref{notation-1} satisfies
\[
\frac d{d\nu }\|C_\nu (\omega )\|^2=2\Big\langle C_\nu (\omega
),\frac d{d\nu }C_\nu (\omega )\Big\rangle\leq 0
\]
and consequently, $R_\nu (\omega )\leq R_1(\omega )$ for $\nu \geq
1$. Hence, $M_{T_1,T_2}^{k,l,\beta }(\nu ,\omega )$ is uniformly
bounded in $\nu $ and
\begin{equation}\label{inequality-1}
\begin{split}
-\beta \nu \|x_\nu ^{(k)}(t)-\ x_\nu ^{(l)}(t)\|^2+y_\nu
^{k,l}(t)\leq \frac 1\nu M_{T_1,T_2}^{k,l,\beta }(\omega)
\end{split}
\end{equation}
uniformly for $t\in [T_1,T_2]$ with
\[
M_{T_1,T_2}^{k,l,\beta }(\omega )=\sup\limits_{\nu \geq
1}M_{T_1,T_2}^{k,l,\beta }(\nu ,\omega ).
\]

Now let us estimate the difference between any two components of a
solution  to the coupled RODEs
\eqref{linear-coupled-rode-again}-\eqref{initial-data} as $\nu
\rightarrow \infty$.

\begin{lemma}\label{key-lemma-1}
The difference between any two components of a solution $\big(x_\nu
^{(1)}(t),x_\nu ^{(2)}(t),...,x_\nu ^{(N)}(t)\big)^{\top}$ of the
coupled RODEs \eqref{linear-coupled-rode-again}-\eqref{initial-data}
vanishes uniformly in any bounded time interval as the coupled
coefficient $\nu$ goes to infinity, namely, for any bounded interval
$[T_1,T_2]$ and any $t\in [T_1,T_2]$,
\[
\lim_{\nu \rightarrow \infty }\|x_\nu ^{(j)}(t)-x_\nu ^{(k)}(t)\|=0
\]
for all $j,\ k\in \big\{1,2,\dots,N\big\}$.
\end{lemma}

\begin{proof}
Equivalently, we can estimate the difference between any two adjacent components only,
where the first and the last component of the solution are considered to be adjacent.
From now on, we call the difference between two components of the
solution a term. In the following process of estimations, we note
that only one new term will be involved in each step which continues
the process, except the last step that ends the process.

 Let us begin our estimations with $x_\nu^{(1)}(t),\ x_\nu^{(2)}(t)$.
\[
\begin{split}
\frac{d}{dt}\|x_\nu^{(1)}(t)-x_\nu^{(2)}(t)\|^2&=2\Big\langle %
x_\nu^{(1)}(t)-x_\nu^{(2)}(t),F^{(1)}(x_\nu^{(1)}(t),O_t^{(1)})-F^{(2)}(x_%
\nu^{(2)}(t),O_t^{(2)})\Big\rangle \\
& \quad +2\Big\langle x_\nu^{(1)}(t)-x_\nu^{(2)}(t),-3\nu\big(%
x_\nu^{(1)}(t)-x_\nu^{(2)}(t)\big)\Big\rangle \\
& \quad +2\Big\langle x_\nu^{(1)}(t)-x_\nu^{(2)}(t),\nu\big(%
x_\nu^{(N)}(t)-x_\nu^{(3)}(t)\big)\Big\rangle \\
&\leq-5\nu\|x_\nu^{(1)}(t)-x_\nu^{(2)}(t)\|^2+\nu\|%
x_\nu^{(3)}(t)-x_\nu^{(N)}(t)\|^2+y_\nu^{1,2}(t) \\
&\leq-\alpha\nu\|x_\nu^{(1)}(t)-x_\nu^{(2)}(t)\|^2+\nu\|%
x_\nu^{(3)}(t)-x_\nu^{(N)}(t)\|^2+\frac{1}{\nu}M_{T_1,T_2}^{1,2,5-%
\alpha}(\omega)
\end{split}
\]
uniformly for $t\in[T_1,T_2]$ by \eqref{inequality-1}. Here, we take
\[
\alpha=%
\begin{cases}1-\cos\frac{N\pi}{N+2}, & N\  \text{is
even},\\1-\cos\frac{(N-1)\pi}{N+1}, & N\  \text{is odd}. \end{cases}
\]
In fact, we can take any $\alpha\in\big(-2\cos\frac{N\pi}{N+2},2\big)$ when $%
N$ is even and any $\alpha\in\big(-2\cos\frac{(N-1)\pi}{N+1},2\big)$
when $N$ is odd.

Note that the above estimations generate
$x_\nu^{(3)}(t)-x_\nu^{(N)}(t)$.
\[
\begin{split}
\frac{d}{dt}\|x_\nu^{(3)}(t)-x_\nu^{(N)}(t)\|^2&=2\Big\langle %
x_\nu^{(3)}(t)-x_\nu^{(N)}(t),F^{(3)}(x_\nu^{(3)}(t),O_t^{(3)})-F^{(N)}(x_%
\nu^{(N)}(t),O_t^{(N)})\Big\rangle \\
& \quad +2\Big\langle x_\nu^{(3)}(t)-x_\nu^{(N)}(t),-2\nu\big(%
x_\nu^{(3)}(t)-x_\nu^{(N)}(t)\big)\Big\rangle \\
& \quad+2\Big\langle x_\nu^{(3)}(t)-x_\nu^{(N)}(t),\nu\big(%
x_\nu^{(2)}(t)-x_\nu^{(1)}(t)\big)\Big\rangle \\
& \quad +2\Big\langle x_\nu^{(3)}(t)-x_\nu^{(N)}(t),\nu\big(%
x_\nu^{(4)}(t)-x_\nu^{(N-1)}(t)\big)\Big\rangle \\
&\leq-2\nu\|x_\nu^{(3)}(t)-x_\nu^{(N)}(t)\|^2+\nu\|%
x_\nu^{(1)}(t)-x_\nu^{(2)}(t)\|^2 \\
& \quad+\nu\|x_\nu^{(4)}(t)-x_\nu^{(N-1)}(t)\|^2+y_\nu^{3,N}(t) \\
&\leq-\alpha\nu\|x_\nu^{(3)}(t)-x_\nu^{(N)}(t)\|^2+\nu\|%
x_\nu^{(1)}(t)-x_\nu^{(2)}(t)\|^2 \\
& \quad +\nu\|x_\nu^{(4)}(t)-x_\nu^{(N-1)}(t)\|^2+\frac{1}{\nu}%
M_{T_1,T_2}^{3,N,2-\alpha}(\omega)
\end{split}
\]
uniformly for $t\in[T_1,T_2]$.

Note that $x_\nu^{(1)}(t)-x_\nu^{(2)}(t)$ has been estimated and $%
x_\nu^{(4)}(t)-x_\nu^{(N-1)}(t)$ is generated. Similarly, we have
\[
\begin{split}
\frac{d}{dt}\|x_\nu^{(4)}(t)-x_\nu^{(N-1)}(t)\|^2 &\leq -\alpha\nu%
\|x_\nu^{(4)}(t)-x_\nu^{(N-1)}(t)\|^2+\nu\|x_\nu^{(3)}(t)-x_%
\nu^{(N)}(t)\|^2 \\
& \quad +\nu\|x_\nu^{(5)}(t)-x_\nu^{(N-2)}(t)\|^2+\frac{1}{\nu}%
M_{T_1,T_2}^{4,N-1,2-\alpha}(\omega)
\end{split}
\]
uniformly for $t\in[T_1,T_2]$.

Continue such estimations, we obtain
\[
\begin{split}
\frac{d}{dt}\|x_\nu^{(j+3)}(t)-x_\nu^{(N-j)}(t)\|^2&\leq-\alpha\nu%
\|x_\nu^{(j+3)}(t)-x_\nu^{(N-j)}(t)\|^2+\nu\|x_\nu^{(j+2)}(t)-x_%
\nu^{(N-j+1)}(t)\|^2 \\
&\quad +\nu\|x_\nu^{(j+4)}(t)-x_\nu^{(N-j-1)}(t)\|^2+\frac{1}{\nu}%
M_{T_1,T_2}^{j+3,N-j,2-\alpha}(\omega)
\end{split}
\]
uniformly for $t\in[T_1,T_2]$, for $j=2,3,\dots$

Now there exists a question: when and where does this process end?
There are two cases:  $N$ is even and $N$ is odd.

\vspace{2mm} \noindent
\begin{bfseries}
Case 1. $N$ is even
\end{bfseries}

Go on the above process with $j$ increasing. When $j=\frac{N}{2}-3$,
we have
\[
\begin{split}
\frac{d}{dt}\|x_\nu^{(\frac{N}{2})}(t)-x_\nu^{(\frac{N}{2}+3)}(t)\|%
^2&\leq-\alpha\nu\|x_\nu^{(\frac{N}{2})}(t)-x_\nu^{(\frac{N}{2}+3)}(t)%
\|^2+\nu\|x_\nu^{(\frac{N}{2}-1)}(t)-x_\nu^{(\frac{N}{2}+4)}(t)\|%
^2 \\
&\quad +\nu\|x_\nu^{(\frac{N}{2}+1)}(t)-x_\nu^{(\frac{N}{2}+2)}(t)\|%
^2+\frac{1}{\nu}M_{T_1,T_2}^{\frac{N}{2},\frac{N}{2}+3,2-\alpha}(\omega)
\end{split}
\]
uniformly for $t\in[T_1,T_2]$.

As $j$ increases to $\frac{N}{2}-2$, we have
\[
\begin{split}
\frac{d}{dt}\|x_\nu^{(\frac{N}{2}+1)}(t)-x_\nu^{(\frac{N}{2}+2)}(t)\|%
^2&\leq-\alpha\nu\|x_\nu^{(\frac{N}{2}+1)}(t)-x_\nu^{(\frac{N}{2}+2)}(t)%
\|^2+\nu\|x_\nu^{(\frac{N}{2})}(t)-x_\nu^{(\frac{N}{2}+3)}(t)\|^2
\\
&\quad +\frac{1}{\nu}M_{T_1,T_2}^{\frac{N}{2}+1,\frac{N}{2}%
+2,5-\alpha}(\omega)
\end{split}
\]
uniformly for $t\in[T_1,T_2]$, which ends this process.

For ease of notation, we rewrite the above inequalities in the matrix form,
\begin{equation}\label{diff-ineq-simple-form-1}
\begin{split}
\dot{\mathbf{y}}(t)\leq \mathbf{A}_\nu\mathbf{y}(t)+\frac{1}{\nu}\mathbf{M}
\end{split}
\end{equation}
uniformly for $t\in[T_1,T_2]$ with two $\frac{N}{2}$-dimensional
vectors
\[
\begin{split}
\mathbf{y}(t)=\Big(\|x_\nu^{(1)}(t)-x_\nu^{(2)}(t)\|^2,\|%
x_\nu^{(3)}(t)-x_\nu^{(N)}(t)\|^2,\dots,\|x_\nu^{(\frac{N}{2}%
+1)}(t)-x_\nu^{(\frac{N}{2}+2)}(t)\|^2\Big)^{\top}, \quad
t\in\mathbb{R},
\end{split}
\]
\[
\begin{split}
\mathbf{M}=\Big(M_{T_1,T_2}^{1,2,5-\alpha}(\omega),M_{T_1,T_2}^{3,N,2-%
\alpha}(\omega),\dots,M_{T_1,T_2}^{\frac{N}{2},\frac{N}{2}%
+3,2-\alpha}(\omega),M_{T_1,T_2}^{\frac{N}{2}+1,\frac{N}{2}%
+2,5-\alpha}(\omega)\Big)^{\top}
\end{split}
\]
and a $\frac{N}{2}\times\frac{N}{2}$ matrix
\[
\begin{split}
\mathbf{A}_\nu=
\begin{pmatrix}
-\alpha\nu    &   \nu             &  0    &  \cdots  & 0 \\
\nu     &  -\alpha\nu   &  \nu     & \ddots  &  \vdots\\
 0        &  \nu       &  \ddots &  \ddots  &   0 \\
\vdots     & \ddots     & \ddots   &  -\alpha\nu  & \nu\\
0   & \cdots  & 0  & \nu  &  -\alpha\nu\end{pmatrix}.
\end{split}
\]

By Lemma \ref{diff-inte-ineq}, it follows from
\eqref{diff-ineq-simple-form-1} that
\begin{equation}\label{inte-ineq-simple-form-1}
\begin{split}
\mathbf{y}(t)\leq e^{(t-t_0)\mathbf{A}_\nu}\mathbf{y}(t_0)+\frac{1}{\nu}%
\int_{t_0}^te^{(t-u)\mathbf{A}_\nu}\mathbf{M}du.
\end{split}
\end{equation}
By Lemma \ref{key-lemma}, $\frac{1}{\nu}\mathbf{A}_\nu$ is negative
definite, then similar to Lemma \ref{estimation-expon-1},
\[
\|e^{\nu(t-t_0)\mathbf{A}_\nu}\mathbf{y}(t_0)\|\leq
e^{\nu(t-t_0)\lambda_{\max}}\|\mathbf{y}(t_0)\|,
\]
where $\lambda_{\max}=-\alpha-2\cos\frac{N\pi}{N+2}<0$ is the
maximal eigenvalue of $\frac{1}{\nu}\mathbf{A}_\nu$. Thus, it
follows from \eqref{inte-ineq-simple-form-1} that
\[
\mathbf{y}(t)\rightarrow0\quad \text{as} \quad \nu\rightarrow\infty
\]
uniformly for $t\in[T_1,T_2]$, which implies that $\|x_\nu^{(1)}(t)-x_%
\nu^{(2)}(t)\|^2$ and $\|x_\nu^{(\frac{N}{2}+1)}(t)-x_\nu^{(\frac{N}{%
2}+2)}(t)\|^2$ tend to $0$ uniformly for $t\in[T_1,T_2]$ as $%
\nu\rightarrow\infty$.

\vspace{2mm}\noindent
\begin{bfseries}
Case 2. $N$ is odd
\end{bfseries}

Similarly, when $j=\frac{N-1}{2}-3$, we have
\[
\begin{split}
\frac{d}{dt}\|x_\nu^{(\frac{N-1}{2})}(t)-x_\nu^{(\frac{N+1}{2}+3)}(t)%
\|^2&\leq-\alpha\nu\|x_\nu^{(\frac{N-1}{2})}(t)-x_\nu^{(\frac{N+1}{2}%
+3)}(t)\|^2+\nu\|x_\nu^{(\frac{N-1}{2}-1)}(t)-x_\nu^{(\frac{N+1}{2}%
+4)}(t)\|^2 \\
&\quad +\nu\|x_\nu^{(\frac{N+1}{2})}(t)-x_\nu^{(\frac{N+1}{2}+2)}(t)%
\|^2+\frac{1}{\nu}M_{T_1,T_2}^{\frac{N-1}{2},\frac{N+1}{2}%
+3,2-\alpha}(\omega)
\end{split}
\]
uniformly for $t\in[T_1,T_2]$.

As $j$ increases to $\frac{N+1}{2}-3$, we have
\[
\begin{split}
\frac{d}{dt}\|x_\nu^{(\frac{N+1}{2})}(t)-x_\nu^{(\frac{N+1}{2}+2)}(t)%
\|^2&\leq-\alpha\nu\|x_\nu^{(\frac{N+1}{2})}(t)-x_\nu^{(\frac{N+1}{2}%
+2)}(t)\|^2+\nu\|x_\nu^{(\frac{N-1}{2})}(t)-x_\nu^{(\frac{N+1}{2}%
+3)}(t)\|^2 \\
&\quad +\frac{1}{\nu}M_{T_1,T_2}^{\frac{N+1}{2},\frac{N+1}{2}%
+2,5-\alpha}(\omega)
\end{split}
\]
uniformly for $t\in[T_1,T_2]$, which ends this process.

We can also rewrite above inequalities in the matrix form
\begin{equation}\label{diff-ineq-simple-form-2}
\begin{split}
\dot{\widetilde{\mathbf{y}}}(t)\leq \widetilde{\mathbf{A}}_\nu\widetilde{%
\mathbf{y}}(t)+\frac{1}{\nu}\widetilde{\mathbf{M}}
\end{split}
\end{equation}
uniformly for $t\in[T_1,T_2]$ with two $\frac{N-1}{2}$-dimensional
vectors
\[
\begin{split}
\widetilde{\mathbf{y}}(t)=\Big(\|x_\nu^{(1)}(t)-x_\nu^{(2)}(t)\|^2,%
\|x_\nu^{(3)}(t)-x_\nu^{(N)}(t)\|^2,\dots,\|x_\nu^{(\frac{N+1}{2}%
)}(t)-x_\nu^{(\frac{N+1}{2}+2)}(t)\|^2\Big)^{\top},\quad
t\in\mathbb{R},
\end{split}
\]
\[
\begin{split}
\widetilde{\mathbf{M}}=\Big(M_{T_1,T_2}^{1,2,5-\alpha}(%
\omega),M_{T_1,T_2}^{3,N,2-\alpha}(\omega),\dots,M_{T_1,T_2}^{\frac{N-1}{2},%
\frac{N+1}{2}+3,2-\alpha}(\omega),M_{T_1,T_2}^{\frac{N+1}{2},\frac{N+1}{2}%
+2,5-\alpha}(\omega)\Big)^{\top}
\end{split}
\]
and a $\frac{N-1}{2}\times\frac{N-1}{2}$ matrix
\[
\begin{split}
\widetilde{\mathbf{A}}_\nu=
\begin{pmatrix}
-\alpha\nu    &   \nu             &  0    &  \cdots  & 0 \\
\nu     &  -\alpha\nu   &  \nu     & \ddots  &  \vdots\\
 0        &  \nu       &  \ddots &  \ddots  &   0 \\
\vdots     & \ddots     & \ddots   &  -\alpha\nu  & \nu\\
0   & \cdots  & 0  & \nu  &  -\alpha\nu\end{pmatrix}.
\end{split}
\]

By Lemma \ref{diff-inte-ineq}, it follows from
\eqref{diff-ineq-simple-form-2} that
\begin{equation}\label{inte-ineq-simple-form-2}
\begin{split}
\widetilde{\mathbf{y}}(t)\leq e^{(t-t_0)\widetilde{\mathbf{A}}_\nu}%
\widetilde{\mathbf{y}}(t_0)+\frac{1}{\nu}\int_{t_0}^te^{(t-u)\widetilde{%
\mathbf{A}}_\nu}\widetilde{\mathbf{M}}du.
\end{split}
\end{equation}
Similar to the case that $N$ is even, it follows from \eqref{inte-ineq-simple-form-2} that $\|%
x_\nu^{(1)}(t)-x_\nu^{(2)}(t)\|^2$ tends to $0$ uniformly for
$t\in[T_1,T_2]$ as $\nu\rightarrow\infty$.

For other adjacent components, the above process can be duplicated. Hence,
after we have dealt with any adjacent components, we can conclude
that the difference between any two components of a solution of
the coupled RODEs (3.1) goes to $0$ uniformly for $t%
\in[T_1,T_2]$ as $\nu\rightarrow\infty$. In fact, if $N$ is even,
another adjacent component will be involved while we focus on current
adjacent
components. For example, $x_\nu^{(\frac{N}{2}+1)}(t)-x_\nu^{(\frac{N}{2}%
+2)}(t)$ is involved while we duplicate with
$x_\nu^{(1)}(t)-x_\nu^{(2)}(t)$. So the above process can be done
for only $\frac{N}{2}$ times if $N$ is even.
\end{proof}

\begin{remark}
In the case of $N=3$, the proof of Lemma \ref{key-lemma-1} can be
simplified since each term is only related to itself.
\end{remark}

Lemma \ref{key-lemma-1} implies that all components of a solution of
\eqref{linear-coupled-rode-again}-\eqref{initial-data} tend to the
same limit uniformly for $t\in [T_1,T_2]$ as $\nu\rightarrow\infty$.
Now, we find what they converge to.

Consider the averaged RODE \eqref{average-rode}
\begin{equation}\label{average-rode-again}
\begin{split}
\frac{dz}{dt}=\frac{1}{N}\sum%
\limits_{j=1}^{N}e^{-O_t^{(j)}}f^{(j)}(e^{O_t^{(j)}}z)+\frac{1}{N}%
\sum\limits_{j=1}^{N}O_t^{(j)}z.
\end{split}
\end{equation}

\begin{lemma}
The random dynamical system $\varphi(t,\omega)$ generated by the
solution of RODE \eqref{average-rode-again} has a singleton sets
random attractor denoted by $\{\bar z(\omega )\}$. Furthermore,
\[
\bar z(\theta _t\omega )\exp \Big(
\frac 1N\sum\limits_{j=1}^NO_t^{(j)}(\omega )\Big)
\]
is the stationary stochastic solution of equivalently averaged SODE
\begin{equation}\label{equivalent-average-sode}
\begin{split}
dZ_t=\frac 1N\sum\limits_{j=1}^Ne^{-\zeta _t^{(j)}}f^{(j)}(e^{\zeta
_t^{(j)}}Z_t)dt+\frac 1N\sum\limits_{i=1}^m\Big(\sum\limits_{j=1}^Nc_i^{(j)}%
\Big)Z_t\circ dW_t^{(i)},
\end{split}
\end{equation}
where $\zeta _t^{(j)}=\frac 1N\sum\limits_{k=1}^N\big(O_t^{(j)}-O_t^{(k)}%
\big),\ j=1,\dots ,N$.
\end{lemma}

\begin{proof}Suppose that $z_1(t)$, $z_2(t)$ are two solutions of \eqref{average-rode-again}. We have
\[
\begin{split}
\frac{d}{dt}\|z_1(t)-z_2(t)\|^2\leq\Big(-2L+\frac{2}{N}%
\sum\limits_{j=1}^{N}O_t^{(j)}\Big)\|z_1(t)-z_2(t)\|^2.
\end{split}
\]
It follows from Gronwall's Lemma that
\[
\begin{split}
\|z_1(t)-z_2(t)\|^2\leq\exp\Bigg(-2t\Bigg(L-\frac{1}{N}\sum\limits_{j=1}^{N}%
\frac{1}{t}\int_{0}^{t}O_t^{j}d\tau\Bigg)\Bigg)\|z_1(0)-z_2(0)\|^2.
\end{split}
\]
Hence, by Lemma \ref{property-OU}, we have
\[
\begin{split}
\lim\limits_{t\rightarrow\infty}\|z_1(t)-z_2(t)\|^2=0,
\end{split}
\]
which means all solutions of \eqref{average-rode-again} converge
pathwise to each other.

Now we use the theory of random dynamical systems to see what they
converge to. Suppose $z(t)$ is a solution of (4.6), we have
\[
\begin{split}
\frac{d}{dt}\big\|z(t)\big\|^2\leq\Big(-L+\frac{2}{N}\sum%
\limits_{j=1}^{N}O_t^{(j)}\Big)\big\|z(t)\big\|^2+\frac{1}{N}%
\sum\limits_{j=1}^{N}\frac{e^{-2O_t^{(j)}}}{L}\big\|f^{(j)}(0)\big\|^2.
\end{split}
\]
It follows from Gronwall's Lemma that
\[
\begin{split}
\|z(t)\|^2 \leq e^{-L(t-t_0)+\frac{2}{N}\sum\limits_{j=1}^{N}%
\int_{t_0}^{t}O_{\tau}^{(j)}d\tau}\|z(t_0)\|^2 +\frac{1}{N}\sum\limits_{j=1}^{N}\frac{\|f^{(j)}(0)\|^2}{L}%
\int_{t_0}^{t}e^{-2O_{u}^{(j)}}e^{-L(t-u)+\frac{2}{N}\sum\limits_{k=1}^{N}%
\int_{u}^{t}O_{\tau}^{(k)}d\tau }du.
\end{split}
\]
Thus, by Lemma \ref{property-OU}, we obtain
\[
\begin{split}
\|z(t)\|^2\leq e^{-\frac{L}{2}(t-t_0)}\|z(t_0)\|^2 +\frac{1}{%
N}\sum\limits_{j=1}^{N}\frac{\|f^{(j)}(0)\|^2}{L}%
\int_{t_0}^{t}e^{-2O_{u}^{(j)}}e^{-L(t-u)+\frac{2}{N}\sum\limits_{k=1}^{N}%
\int_{u}^{t}O_{\tau}^{(k)}d\tau }du
\end{split}
\]
for $-t_0,t>T_{\omega}$.

By pathwise pullback convergence with $t_0\rightarrow-\infty$, the random
closed ball centered at the origin with random radius
$R(\omega)$ is a pullback absorbing set of $\varphi(t,\omega)$ in
$\mathcal{D}$ for $t>T_{\omega}$, where
\[
\begin{split}
R^2(\omega)=1+\frac{1}{N}\sum\limits_{j=1}^{N}\frac{\|f^{(j)}(0)\|^2%
}{L}\int_{-\infty}^{0}e^{Lu-2O_{u}^{(j)}}e^{\frac{2}{N}\sum\limits_{k=1}^{N}%
\int_{u}^{0}O_{\tau}^{(k)}d\tau }du.
\end{split}
\]
Note that the integrals on the right-hand side are well defined by
Lemma \ref{property-OU}.

By Theorem 4.1 in [9], there exists a random attractor
$\{\bar{z}(\omega)\}$ for $\varphi(t,\omega)$. Since all solutions
of \eqref{average-rode-again} converge pathwise to each other, the
random attractor $\{\bar{z}(\omega)\}$ are composed of singleton
sets.

Note that the averaged RODE \eqref{average-rode-again} is
transformed from the averaged SODE \eqref{equivalent-average-sode}
by the transformation
\[
\begin{split}
z(t,\omega)=\exp\Big(-\frac{1}{N}\sum\limits_{j=1}^{N}O_t^{(j)}(\omega)\Big)Z_t(%
\omega),
\end{split}
\]
so the pathwise singleton sets attractor $\bar{z}(\theta_t\omega)\exp\Big(%
\frac{1}{N}\sum\limits_{j=1}^{N}O_t^{(j)}(\omega)\Big)$ is a
stationary solution of the averaged SODE
\eqref{equivalent-average-sode} since the Ornstein-Uhlenbeck process
is stationary.
\end{proof}

We now show another main result of this paper.

\begin{theorem}\label{main-result}
Let
\[
\begin{split}
\Big(\bar x_{\nu _n}^{(1)}(t,\omega ),\bar x_{\nu _n}^{(2)}(t,\omega
),\dots ,\bar x_{\nu _n}^{(N)}(t,\omega )\Big)^{\top}=\Big(\bar
x_{\nu _n}^{(1)}(\theta _t\omega ),\bar x_{\nu _n}^{(2)}(\theta
_t\omega ),\dots ,\bar x_{\nu _n}^{(N)}(\theta _t\omega
)\Big)^{\top}
\end{split}
\]
be the singleton sets random attractor of the random dynamical
system $\phi(t,\omega)$ generated by the solution of RODEs
\eqref{linear-coupled-rode-again}-\eqref{initial-data}, then
\[
\begin{split}
\Big(\bar x_{\nu _n}^{(1)}(t,\omega ),\bar x_{\nu _n}^{(2)}(t,\omega
),\dots ,\bar x_{\nu _n}^{(N)}(t,\omega )\Big)^{\top}\rightarrow
\Big(\bar z(t,\omega ),\bar z(t,\omega ),\dots ,\bar z(t,\omega
)\Big)^{\top}
\end{split}
\]
pathwise uniformly for $t\in [T_1,T_2]$ for any sequence $\nu
_n\rightarrow \infty $, where $\bar z(t,\omega )=\bar z(\theta
_t\omega )$ solves the averaged RODE \eqref{average-rode-again} and
$\bar z(\omega)$ is the singleton sets random attractor of the
random dynamical system $\varphi(t,\omega)$ generated by the
solution of the averaged RODE \eqref{average-rode-again}.
\end{theorem}

\begin{proof}
Define
\[
\begin{split}
\bar{z}_\nu(\omega)=\frac{1}{N}\sum\limits_{j=1}^{N}\bar{x}_\nu^{(j)}(\omega),
\end{split}
\]
where
$\Big\{\Big(\bar{x}_\nu^{(1)}(\omega),\bar{x}_\nu^{(2)}(\omega),\dots,\bar{x}_\nu^{(N)}(\omega)\Big)\Big\}$
is the singleton sets random attractor of the random dynamical
system generated by RODEs
\eqref{linear-coupled-rode-again}-\eqref{initial-data}. Thus,
$\bar{z}_\nu(t,\omega)=\bar{z}_\nu(\theta_t\omega)$ satisfies
\begin{equation}\label{equality-1}
\begin{split}
\frac{d}{dt}\bar{z}_\nu(t,\omega)=\frac{1}{N}\sum\limits_{j=1}^{N}\Big(e^{-O_t^{(j)}(\omega)}f^{(j)}(e^{O_t^{(j)}(\omega)}\bar{x}_\nu^{(j)}(t,\omega))+O_t^{(j)}(\omega)\bar{x}_\nu^{(j)}(t,\omega)\Big).
\end{split}
\end{equation}

Note that
\[
\begin{split}
\|\frac{d}{dt}\bar{z}_\nu(t,\omega)\|^2\leq\frac{2}{N}\sum\limits_{j=1}^{N}\Big(e^{-2O_t^{(j)}(\omega)}\|f^{(j)}(e^{O_t^{(j)}(\omega)}\bar{x}_\nu^{(j)}(t,\omega))\|^2
+|O_t^{(j)}(\omega)|^2\|\bar{x}_\nu^{(j)}(t,\omega)\|^2\Big),
\end{split}
\]
by continuity and the fact that these solutions belong to the
compact ball $B_1(\omega)$, it follows that
$$\sup\limits_{t\in[T_1,T_2]}\|\frac{d}{dt}\bar{z}_\nu(t,\omega)\|
\leq\Bigg(\frac{2}{N}\sum\limits_{j=1}^{N}\frac{\beta}{4}M_{T_1,T_2}^{j,\bullet,\beta}
(\omega)\Bigg)^{\frac{1}{2}}<\infty.
$$
By Ascoli-Arzel\`{a} Theorem,  there exists a subsequence
$\nu_{n_k}\rightarrow\infty$ such that
$\bar{z}_{\nu_{n_k}}(t,\omega)$ converges to $\bar{z}(t,\omega)$ as
$n_k\rightarrow\infty$.

Since difference between any two components of a solution of the
coupled RODEs \eqref{linear-coupled-rode-again} tends to $0$
uniformly for $t\in[T_1,T_2]$ as $\nu\rightarrow\infty$, we have
\[
\begin{split}
\bar{x}_{\nu_{n_k}}^{(j)}(t,\omega)&=N\bar{z}_{\nu_{n_k}}(t,\omega)-\sum\limits_{j'\neq
j}\bar{x}_{\nu_{n_k}}^{(j')}(t,\omega)\\
&=\bar{z}_{\nu_{n_k}}(t,\omega)+\sum\limits_{j'\neq
j}\Big(\bar{z}_{\nu_{n_k}}(t,\omega)-\bar{x}_{\nu_{n_k}}^{(j')}(t,\omega)\Big)\\
&=\bar{z}_{\nu_{n_k}}(t,\omega)+\frac{1}{N}\sum\limits_{j'\neq
j}\sum\limits_{j''\neq
j'}\Big(\bar{x}_{\nu_{n_k}}^{(j'')}(t,\omega)-\bar{x}_{\nu_{n_k}}^{(j')}(t,\omega)\Big)\\
&\rightarrow\bar{z}(t,\omega)\\
\end{split}
\]
uniformly for $t\in[T_1,T_2]$ as $\nu_{n_k}\rightarrow\infty$ for
$j=1,\dots,N$.

Furthermore, it follows from \eqref{equality-1} that
\[
\begin{split}
\bar{z}_\nu(t,\omega)&=\bar{z}_\nu(T_1,\omega)+\frac{1}{N}\sum\limits_{j=1}^{N}\int_{T_1}^{t}e^{-O_s^{(j)}(\omega)}f^{(j)}(e^{O_s^{(j)}(\omega)}\bar{x}_\nu^{(j)}(s,\omega))ds\\
&\quad
+\frac{1}{N}\sum\limits_{j=1}^{N}\int_{T_1}^{t}O_s^{(j)}(\omega)\bar{x}_\nu^{(j)}(s,\omega)ds
\end{split}
\]
Thus,
\[
\begin{split}
\bar{z}(t,\omega)&=\bar{z}(T_1,\omega)+\frac{1}{N}\sum\limits_{j=1}^{N}\int_{T_1}^{t}e^{-O_s^{(j)}(\omega)}f^{(j)}(e^{O_s^{(j)}(\omega)}\bar{z}(s,\omega))ds\\
&\quad
+\frac{1}{N}\sum\limits_{j=1}^{N}\int_{T_1}^{t}O_s^{(j)}(\omega)\bar{z}(s,\omega)ds
\end{split}
\]
uniformly for $t\in[T_1,T_2]$ as $\nu_{n_k}\rightarrow\infty$, which
means that $\bar{z}(t,\omega)$ solves RODE
\eqref{average-rode-again}.

Note that any possible subsequences converge to the same limit, so
every sequence $\bar{z}_{\nu_n}(t,\omega)$ converges to
$\bar{z}(t,\omega)$ uniformly for $t\in[T_1,T_2]$ as
$\nu_n\rightarrow\infty$ by Lemma 2.2 in \cite{8}.

Finally, since the random dynamical system generated by the solution
of RODE \eqref{average-rode-again} has a singleton sets random
attractor $\{\bar{z}(\omega)\}$, the stationary stochastic process
$\bar{z}(\theta_t\omega)$ must be equal to $\bar{z}(t,\omega)$,
namely $\bar{z}(t,\omega)=\bar{z}(\theta_t\omega)$.
\end{proof}

As a straightforward consequence of Theorem \ref{main-result}, we
have

\begin{corollary}
$\Big(\bar x_\nu ^{(1)}(t,\omega ),\bar x_\nu ^{(2)}(t,\omega
),\dots ,\bar x_\nu ^{(N)}(t,\omega )\Big)^{\top}\rightarrow
\Big(\bar z(t,\omega ),\bar z(t,\omega ),\dots ,\bar z(t,\omega
)\Big)^{\top}$ pathwise uniformly for $t\in [T_1,T_2]$ as $\nu
\rightarrow \infty $.
\end{corollary}

In terms of the coupled SODEs \eqref{equivalent-sode}, its
stationary stochastic solution
\[
\begin{split}
\Big(\bar{x}^{(1)}_\nu(\theta_t\omega)e^{O_t^{(1)}(\omega)},\bar{x}%
^{(2)}_\nu(\theta_t\omega)e^{O_t^{(2)}(\omega)},\dots,\bar{x}%
^{(N)}_\nu(\theta_t\omega)e^{O_t^{(N)}(\omega)}\Big)^{\top}
\end{split}
\]
tends pathwisely to
\[
\begin{split}
\Big(\bar{z}(\theta_t\omega)e^{O_t^{(1)}(\omega)},\bar{z}(\theta_t%
\omega)e^{O_t^{(2)}(\omega)},\dots,\bar{z}(\theta_t\omega)e^{O_t^{(N)}(%
\omega)}\Big)^{\top}
\end{split}
\]
uniformly for $t\in[T_1,T_2]$ as $\nu\rightarrow\infty$. Obviously, if $%
c_{i}^{(1)}=c_{i}^{(2)}=\dots=c_{i}^{(N)}=c_{i}$ for $i=1,\dots,m$
in \eqref{equivalent-sode}, i.e. the driving noise is the same,
exact synchronization of solutions of the coupled SODEs
\[
\begin{split}
dX^{(j)}_t &=\bigg(f^{(j)}(X^{(j)}_t)+\nu\big(%
X^{(j-1)}_t-2X^{(j)}_t+X^{(j+1)}_t\big)\bigg)dt \\
&\quad +\sum\limits_{i=1}^mc_{i}X^{(j)}_t\circ dW_t^{(i)},\quad j=1,\dots,N
\end{split}
\]
occurs.

\begin{remark}
The results in this paper hold just in the almost everywhere sense
because $\omega\in\overline{\Omega}$ here (see Lemma
\ref{property-OU} and some interpretations below the lemma).
\end{remark}

\vskip 0.5cm

\noindent \textbf{Acknowledgement}

\vskip 0.2cm

\noindent Partially supported by the Leading Academic Discipline
Project of Shanghai Normal University (No. DZL707), the National
Natural Science Foundation of China under Grant No. 10771139, the
National Ministry of Education of China (200802700002), the
Foundation of Shanghai Talented Persons (No.049), the Innovation
Program of Shanghai Municipal Education Commission under Grant No.
08ZZ70, Foundation of Shanghai Normal University under Grant
DYL200803. The second author would like to express his sincere
thanks to Professor Wenxian Shen and Professor T. Caraballo for
their helpful discussion and kind help.

%-----------------------------------------------------------------------------------------------------

\end{document}